\documentclass[12pt,a4paper]{article}
\usepackage[top=2.5cm,bottom=2.5cm,left=2.5cm,right=2.5cm]{geometry}

\usepackage{amssymb,amsmath,graphicx,color,amsthm,centernot}
\usepackage[capitalise]{cleveref}
\usepackage[labelformat=simple]{subcaption}
\usepackage[font=small,labelfont=bf,width=12.5cm]{caption}

\newtheorem{claim}{Claim}
\newtheorem{theorem}{Theorem}

\newtheorem{corollary}{Corollary}

\newtheorem{conjecture}{Conjecture}

\newtheorem{problem}{Problem}

\newcommand{\set}[1]{\ensuremath{\left\{#1 \right\}}}

\newcommand{\chis}[1]{\ensuremath{\chi_{s}'(#1)}}
\newcommand{\chiss}[1]{\ensuremath{\chi_{\mathrm{ss}}'(#1)}}

\newcommand{\F}[1]{\ensuremath{A(#1)}}
\newcommand{\FG}[1]{\ensuremath{A_\sigma(#1)}}
\newcommand{\C}[1]{\ensuremath{C(#1)}}

\newenvironment{proofclaim}[1][]%
    {\noindent \emph{Proof.} {}{#1}{}}{$~$\hfill $~\blacklozenge$ \vspace{0.2cm}}

\definecolor{defblue}{rgb}{0.4,0,0.84}
\definecolor{greyblue}{rgb}{0.23,0.4,0.70}
\definecolor{orange}{rgb}{1.0,0.5,0.2}
\definecolor{violet}{rgb}{0.55,0,0.55}

\usepackage{url}
\makeatletter
\g@addto@macro{\UrlBreaks}{\UrlOrds}
\makeatother

\begin{document}

\title{{\bf Revisiting Semistrong Edge-Coloring \\of Graphs}}

\author
{
	Borut Lu\v{z}ar\thanks{Faculty of Information Studies in Novo mesto, Slovenia. 
		E-Mail: \texttt{borut.luzar@gmail.com}},\thanks{University of Ljubljana, Faculty of Mathematics and Physics, Slovenia.} \
	Martina Mockov\v{c}iakov\'{a}\thanks{Faculty of Applied Sciences, University of West Bohemia, Pilsen, Czech Republic.}, \
	Roman Sot\'{a}k\thanks{Faculty of Science, Pavol Jozef \v Saf\'{a}rik University, Ko\v{s}ice, Slovakia.
		E-Mail: \texttt{roman.sotak@upjs.sk}}
}

\maketitle

{
\begin{abstract}
	A matching $M$ in a graph $G$ is {\em semistrong} 
	if every edge of $M$ has an endvertex of degree one in the subgraph induced by the vertices of $M$.
	A {\em semistrong edge-coloring} of a graph $G$ is a proper edge-coloring 
	in which every color class induces a semistrong matching.
	In this paper, we continue investigation of properties of semistrong edge-colorings initiated by Gy\'{a}rf\'{a}s and Hubenko 
	({Semistrong edge coloring of graphs}. \newblock {\em J. Graph Theory}, 49 (2005), 39--47). 
	We establish tight upper bounds for general graphs and for graphs with maximum degree $3$. 
	We also present bounds about semistrong edge-coloring which follow from results regarding other,
	at first sight non-related, problems.
	We conclude the paper with several open problems.
\end{abstract}
}


\section{Introduction}

A {\em proper edge-coloring} of a simple graph $G$ is an assignment of colors to its edges such that adjacent edges receive distinct colors.
A proper edge-coloring can also be seen as a decomposition of the edge set in a set of matchings 
in which every matching represents the edges of one color.
The least integer $k$ for which $G$ admits a proper edge-coloring with $k$ colors
is called the {\em chromatic index} of $G$ and denoted $\chi'(G)$.
On the other hand, an edge-coloring in which edges of the same color comprise an {\em induced matching} (also called a {\em strong matching}),
i.e., a set of edges such that the distance between endvertices of any two edges is at least $2$,
is a {\em strong edge-coloring} of $G$.
The least integer $k$ for which $G$ admits a strong edge-coloring with $k$ colors
is called the {\em strong chromatic index} of $G$ and denoted $\chis{G}$.

While Vizing~\cite{Viz64} proved that the chromatic index of a graph $G$ attains one of just two values, $\Delta(G)$ or $\Delta(G)+1$, 
the interval of possible values for $\chis{G}$ is much wider;
namely, between $\Delta(G)$ and (at least) $\frac{5}{4}\Delta(G)^2$.
Note that the lower bound increases to $2\Delta(G)-1$ if there are two vertices of maximum degree adjacent in $G$.
The complete classification of regular graphs attaining this lower bound was just recently obtained in~\cite{LuzMacSkoSot22}.
The upper bound $\frac{5}{4}\Delta(G)^2$ was conjectured by Erd\H{o}s and Ne\v{s}et\v{r}il in 1985 (see~\cite{Erd88}),
who also provided constructions of graphs attaining this upper bound.
Despite many efforts (see, e.g.,~\cite{BonPerPost22,BruJoo18,MolRee97}), 
the best known upper bound to date is $1.772\Delta(G)^2$ for large enough $\Delta(G)$ due to Hurley et al.~\cite{HurJoaKan21}.

Apart from the general case, determining strong chromatic index for special classes of graphs received
a lot of attention and therefore it is not surprising that relaxed variants of strong edge-colorings appeared.
Roughly, we can divide them into three types: 
\begin{itemize}
	\item[(A)] improper strong edge-colorings, where some edges at distance $1$ or $2$ from an edge can receive the same color (see, e.g.,~\cite{HeLin17});
	\item[(B)] $(1^a,2^b)$-packing edge-colorings, where the set of edges is decomposed into at most $a$ matchings and at most $b$ induced matchings (see, e.g.,~\cite{HocLajLuz22});
	\item[(C)] edge-colorings in which edges of every color induce a matching with particular properties (see, e.g.,~\cite{BasRau18,GodHedHedLas05,GyaHub05}).
\end{itemize}

Although it might not seem so at first sight, the edge-colorings of all three types are very much related.
In this paper, by continuing the work of Gy\'{a}rf\'{a}s and Hubenko initiated in~\cite{GyaHub05},
we consider properties of the semistrong edge-coloring, which can be classified as type~C.
A {\em semistrong edge-coloring} is a proper edge-coloring in which the edges of every color class induce a semistrong matching,
where a matching $M$ of a graph $G$ is {\em semistrong} if every edge of $M$ has an end-vertex of degree one in the induced subgraph $G[V(M)]$.
The least integer $k$ such that $G$ admits a semistrong edge-coloring with at most $k$ colors is 
the {\em semistrong chromatic index} of $G$ and denoted by $\chiss{G}$.

The above definition does not allow existence of a semistrong edge-coloring in any graph with parallel edges, 
since both endvertices of any parallel edge $e$ have degree more than $1$ in the graph induced by $e$.
Therefore, in multigraphs, we only require, for a proper edge-coloring of a multigraph $G$ to be also semistrong,
that the edges of every color class induce a semistrong matching in the underlying graph of $G$, 
i.e., the graph, in which every two vertices adjacent in $G$ are connected by exactly one edge.

Clearly, every strong matching is also semistrong, and so, for every graph $G$, we have
$$
	\nu_{s}(G) \le \nu_{\mathrm{ss}}(G) \le \nu(G)\,,
$$
where by $\nu(G)$, $\nu_{s}(G)$, and $\nu_{\mathrm{ss}}(G)$ 
we denote the maximum size of a matching, strong matching, and semistrong matching of $G$, respectively. 
Consequently,
$$
	\chi'(G) \le \chiss{G} \le \chis{G}\,.
$$
Although the above series of inequalities is trivial, 
no better general upper bound than the bound for the strong edge-coloring has been known for the semistrong edge-coloring of general graphs.
In this paper, we prove the following.
\begin{theorem}
	\label{thm:main}
	For every graph $G$, we have
	$$
		\chiss{G} \le \Delta(G)^2\,.
	$$
\end{theorem}

Let us note that there are graphs $G$ for which $\chiss{G} = \chis{G}$,
e.g., the complete and the complete bipartite graphs. 
The equality for the two classes is a simple consquence of the fact 
that in a semistrong edge-coloring the edges of every $4$-cycle must be colored with four distinct colors.
Therefore, 
\begin{align*}
	\chiss{K_n} = \chis{K_n} = {n \choose 2}
\end{align*}
and 
\begin{align}
	\label{eq:bip}
	\chiss{K_{m,n}} = \chis{K_{m,n}} = m \cdot n\,.
\end{align}

In~\cite{GyaHub05}, the authors found two additional families of graphs with
the same value of strong and semistrong chromatic indices;
namely the Kneser and the subset graphs.
For two positive integers $n$ and $k$, with $k \le n$, a {\em Kneser graph $K(n,k)$} is a graph 
whose vertex set consists of all $k$-subsets of an $n$-element set, 
and two vertices are connected if and only if the corresponding sets are disjoint.
A {\em subset graph} $S_n(k,\ell)$, 
for $0 \le k \le \ell \le n$, 
is a bipartite graph where the partition vertex sets are the $k$- and $\ell$-subsets of the $n$ element set, 
and two vertices (subsets) are connected if and only if one of them is contained in the other. 
Extending the strong edge-coloring results for the two classes due to~\cite{FauGyaSchTuz90} and~\cite{QuiBen97},
Gy\'{a}rf\'{a}s and Hubenko proved the following.
\begin{theorem}[Gy\'{a}rf\'{a}s \& Hubenko~\cite{GyaHub05}]~
	\begin{itemize}
	\item[$(A)$]
			For every Kneser graph $K(n,k)$ it holds 
			$\chi_{s}'(K(n,k)) =\chi_{ss}'(K(n,k)) = {n \choose 2k}.$
	\item[$(B)$]
			For every subset graph $S_n(k,\ell)$ it holds 
			$\chi_{s}'(S_n(k,\ell)) = \chi_{ss}'(S_n(k,\ell)) = {n \choose \ell-k}.$		
	\end{itemize}
\end{theorem}

Additionally, the above authors considered the equality of the two invariant for the $n$-dimensional cubes
and conjectured that $\nu_{s}(Q_n) = \nu_{\mathrm{ss}}(Q_n) = 2^{n-2}$, for every $n \ge 2$.
As noted by Gregor~\cite{Gre20}, the conjecture was established by Diwan~\cite{Diw19}, who in fact considered the problem of the minimum forcing number,
most likely being unaware of resolving another conjecture.
\begin{theorem}[Diwan~\cite{Diw19}]
	\label{conj:hyper}
	For every integer $n \ge 2$, we have
	$$
		\nu_{s}(Q_n) = \nu_{\mathrm{ss}}(Q_n) = 2^{n-2}\,.
	$$
\end{theorem}
Furthermore, in~\cite{FauGyaSchTuz90} it was proved that $\nu_{s}(Q_n)=2^{n-2}$ and $\chis{Q_n}=2n$,
from which it can be concluded that the following holds.
\begin{corollary}[Diwan~\cite{Diw19} \& Faudree et al.~\cite{FauGyaSchTuz90}]
	For every integer $n \ge 2$, we have
	$$
		\chiss{Q_n} = \chis{Q_n} = 2n\,.
	$$
\end{corollary}

Another graph family for which the semistrong chromatic index is completely determined are trees.
The strong chromatic index of a tree $T$ is at most $2\Delta(T)-1$~\cite{FauGyaSchTuz90}.
In the semistrong setting, the bound is much lower as follows from the result due to He and Lin~\cite{HeLin17}
who considered the {\em $(s,t)$-relaxed strong edge-coloring}, 
i.e., an edge-coloring, in which, for every edge $e$ of a graph $G$, the number of edges adjacent to $e$ having the same color as $e$ is at most $s$, 
and the number of edges at distance $2$ from $e$ having the same color as $e$ is at most $t$. 
The corresponding chromatic index is {\em $(s, t)$-relaxed strong chromatic index}, denoted by $\chi'_{(s, t)}(G)$.
They proved that, for every tree $T$, if $s=0$ and $t = \Delta(T) - 1$, then $\chi'_{(s,t)}(T) \le \Delta(T) + 1$ \cite[Lemma~5.1]{HeLin17}.
For the proof, they provided a construction \cite[Algorithm~2]{HeLin17} of an edge-coloring using at most $\Delta(T) + 1$ colors
that is also a semistrong edge-coloring, and thus they also proved the following.
\begin{corollary}[He \& Lin~\cite{HeLin17}]
	For every tree $T$ it holds
	$$
		\chiss{T} \le \Delta(T) + 1\,.
	$$
	Moreover, if $T$ has diameter at most $4$, then
	$$
		\chiss{T} = \Delta(T)\,.
	$$
\end{corollary}

In this paper, we also present a tight result for graphs with maximum degree $3$.
From Theorem~\ref{thm:main} it already follows that at most $9$ colors are needed and that the bound is attained by $K_{3,3}$.
We improve this bound as follows.
\begin{theorem}
	\label{thm:cubic}
	For every connected graph $G$ with maximum degree $3$, distinct from $K_{3,3}$, we have
	$$
		\chiss{G} \le 8\,.
	$$
\end{theorem}

Note that the bound in Theorem~\ref{thm:cubic} is tight, 
since the semistrong chromatic index of the $5$-prism is $8$. 
This follows from the fact that the size of any maximum semistrong matching in the $5$-prism is $2$, 
while it has $15$ edges.
\begin{figure}[htp!]
	$$
		\includegraphics{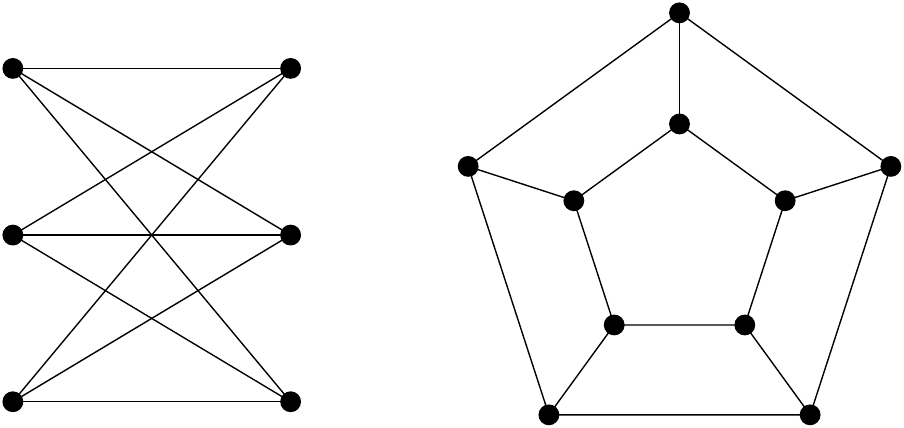}
	$$
	\caption{The semistrong chromatic index of $K_{3,3}$ and the $5$-prism is $9$ and $8$, respectively.}
	\label{fig:prism}
\end{figure}

The structure of the paper is the following.
We begin by presenting notation, definitions and auxiliary results in Section~\ref{sec:prel}.
In Sections~\ref{sec:gen} and~\ref{sec:cubic}, we prove Theorems~\ref{thm:main} and~\ref{thm:cubic}, respectively,
and finally, in Section~\ref{sec:con}, we discuss several additional edge-colorings related to the semistrong edge-coloring
and conclude with some open problems and suggestions for further work.

\section{Preliminaries}
\label{sec:prel}

In this section, we introduce terminology, notation, and some auxiliary result.

When constructing a semistrong edge-coloring with at most $k$ colors, we always assume that the colors are taken 
from the set of the first $k$ positive integers; as customary, we write $[k] = \set{1,\dots,k}$.
We also abuse the notation and denote the set of colors appearing on the edges from a set $X$ 
in a semistrong edge-coloring $\sigma$ as $\sigma(X)$.

Given a partial semistrong edge-coloring $\sigma$, 
a color $\alpha$ is \textit{available} for an edge $e$ if there is no edge at distance at most $2$ from $e$ colored with $\alpha$.
The set of available colors for an edge $e$ is denoted $\FG{e}$, or simply, $\F{e}$. 
A color $\alpha$ is \textit{forbidden} for an edge $e$ if coloring $e$ with $\alpha$ violates the assumptions of the semistrong edge-coloring.

With sets of available colors defined,
in several cases, we will use the following application of Hall's Marriage Theorem~\cite{Hal35}.
\begin{theorem}
	\label{thm:hall}
	Let $G$ be a graph and $\sigma$ a partial semistrong edge-coloring of $G$.
	Let $X = \set{e_1,\dots,e_k}$ be the set of non-colored edges of $G$. Let
	$\mathcal{F} = \set{\F{e_1},\dots,\F{e_k}}$. If for every subset $\mathcal{X} \subseteq \mathcal{F}$ it holds that
	$$
		|\mathcal{X}| \le  \Big | \bigcup_{X \in \mathcal{X}} X \Big |\,,
	$$
	then one can choose an available color for every edge in $X$ such that all the edges receive distinct colors.
\end{theorem}

By $\C{v}$ we denote the set of colors of the edges incident with $v$. 
Throughout the paper, we sometimes simply write `coloring' instead of `semistrong edge-coloring'.

By $N(uv)$, we denote the set of edges adjacent to $uv$, i.e., the edge-neighborhood of $uv$,
and by $N^2(uv)$, we denote the set of edges at distance $1$ or $2$ from $uv$, i.e., the $2$-edge-neighborhood of $uv$.
Similarly, by $N_u(uv)$, we denote the set of edges adjacent to $uv$ with $u$ being one of their endvertices,
and by $N_u^2(uv)$, we denote the set of edges at distance $1$ or $2$ from $uv$ such that 
every edge $e$ in $N_u^2(uv)$ has $u$ as an endvertex or there is an edge having $u$ as endvertex, connecting $e$ and $uv$.


\section{Proof of Theorem~\ref{thm:main}}
\label{sec:gen}

In this section, we prove a stronger result than Theorem~\ref{thm:main}; 
namely, we establish the following.
\begin{theorem}
	\label{thm:s01}
	For every graph $G$, we have
	$$
		\chi'_{(0,1)}(G) \le \Delta(G)^2\,.
	$$
	Moreover, there is always a $(0,1)$-relaxed strong edge-coloring of $G$ using at most $\Delta(G)^2$ colors
	such that the edges of every color induce a semistrong matching.
\end{theorem}

Clearly, Theorem~\ref{thm:main} is a direct corollary of Theorem~\ref{thm:s01}.
Note however that $(0,1)$-relaxed strong edge-coloring without additional conditions 
is not equivalent to the semistrong edge-coloring, since in the former, e.g., non-rainbow $4$-cycles can appear.

\begin{proof}
	For every edge $e = uv \in E(G)$, 
	let $L(e) = (N_u^2(e) \cap N_v^2(e))\setminus N(e)$ be the set of all edges at distance $2$ from $e$ lying is a common $4$-cycle with $e$.
	Additionally, let $k(e) = |N(e)|$ and $\ell(e) = |L(e)|$.
	Observe that $k(e) \le 2(\Delta(G) - 1)$, $\ell(e) \le (\Delta(G) - 1)^2$ 
	and $|N^2(e)| \le 2(\Delta(G)-1)\Delta(G) - \ell(e)$.

	Let $\sigma(G)$ be a proper edge-coloring of $G$ with at most $\Delta(G)^2$ colors 
	such that every edge $e$ receives a color distinct from all colors of edges in $L(e)$.
	Note that such a coloring always exists, since there are at most 
	$k(e) + \ell(e) \le \Delta(G)^2 - 1$ 
	conflicts for every edge $e$,
	and therefore one can, e.g., simply take a greedy approach to find it.

	We now proceed by a contradiction. 
	Among all possible above described colorings, let $\sigma(G)$ be a coloring with the minimum number of {\em distance-$2$ conflicts},
	i.e., pairs of edges at distance $2$ receiving the same colors.
	Denote the number of distance-$2$ conflicts for the coloring $\sigma(G)$ by $\iota(\sigma)$.
	Observe that if every edge $e$ has at most one distance-$2$ conflict, then $\sigma(G)$ is a semistrong edge-coloring.
	Thus, we may assume that there exists an edge $e'$ having at least two distance-$2$ conflicts.

	There are at most $k(e') + \ell(e') \le \Delta(G)^2 - 1$ colors that cannot be used for $e'$.
	Moreover, if there is some other color $\alpha$ used at most once on the edges from $N^2(e') \setminus (N(e') \cup L(e'))$, 
	then we recolor $e'$ with $\alpha$ and so decrease $\iota(\sigma)$, a contradiction.
	Therefore, the number of colors in $N^2(e')$ is at most
	\begin{align*}
		k(e') + \ell(e') + \frac{1}{2}\Big(|N^2(e')| - k(e') - \ell(e')\Big) &\le \Delta(G)^2 - 1\,.
	\end{align*}
	So, we can recolor the edge $e'$ with an available color and decrease $\iota(\sigma)$, a contradiction.
\end{proof}

Let us note here that the above proof implies that less than $\Delta(G)^2$ colors 
suffice for every graph in which no edge has both endvertices of maximum degree.

\section{Proof of Theorem~\ref{thm:cubic}}
\label{sec:cubic}

In this section, we improve the upper bound obtained in the previous section for the class of graphs with maximum degree $3$.
\begin{proof}[Proof of Theorem~\ref{thm:cubic}]
	We prove the theorem by a contradiction. 
	Suppose that $G$ is the minimal counterexample, i.e., a connected subcubic graph distinct from $K_{3,3}$ with $\chiss{G} = 9$ 
	and minimum number of edges among all such graphs; this means that $G$ has at least $9$ edges. 
	We continue by establishing additional structural properties of $G$.
	
	\begin{claim}
		\label{cl:multi}
		$G$ does not contain parallel edges.
	\end{claim}
	
	\begin{proofclaim}
		Suppose to the contrary that there are two edges, $e_1$ and $e_2$, both connecting vertices $u$ and $v$ in $G$. 
		By the minimality of $G$, there exists a semistrong edge-coloring $\sigma$ of $G' = G \setminus \{e_1\}$ with at most $8$ colors.
		Note that $\sigma$ is a partial semistrong edge-coloring of $G$ (with only $e_1$ being non-colored),
		since the distances between the edges in $G'$ are the same as the distances between the edges in $G$.
		Since there are at most $7$ edges in $N^2(e_1)$, there is at least $1$ available color for $e_1$.
		Therefore, we can color $e_1$, and hence extend $\sigma$ to all edges of $G$, a contradiction.
	\end{proofclaim}

	\begin{claim}
		\label{cl:bridge}
		$G$ is $2$-connected.
	\end{claim}
	
	\begin{proofclaim}
		Suppose to the contrary that there is a cut-vertex $v$ in $G$.
		In the setting of subcubic graphs this also means that there is a bridge $uv$ in $G$.
		Let $G_u$ (resp., $G_v$) be the component of $G\setminus \set{uv}$ containing $u$ (resp., $v$).
		By the minimality, there is a semistrong edge-coloring $\sigma_u$ of $G_u$ with at most $8$ colors, 
		and similarly, there is a semistrong edge-coloring $\sigma_v$ of $G_v$ with at most $8$ colors.
		
		The colorings $\sigma_u$ and $\sigma_v$ induce a partial proper edge-coloring $\sigma$ of $G$ with only $uv$ being non-colored.
		Note that $\sigma$ might not be semistrong, since the colors of the edges incident with $u$ and $v$ may be in conflict.
		Therefore, we permute the colors of $\sigma_v$ in such a way that		
		$\sigma(N_v(uv))\cap\sigma(N_u(uv))=\emptyset$ and $|\sigma(N^2(uv))| \le 6$. 
		Note that this can be done, since there are at most four edges in $N^2_v(uv)\setminus N_v(uv)$. 
		This means that there is an available color for the edge $uv$,
		and hence we can extend $\sigma$ to all the edges of $G$, a contradiction.
	\end{proofclaim}
	
	\begin{claim}
		\label{cl:33}
		There are no adjacent triangles in $G$.
	\end{claim}
	
	\begin{proofclaim}
		Clearly, $G$ is not isomorphic to $K_4$ as it admits a strong edge-coloring with $6$ colors.
		Therefore, we may assume, for a contradiction, that there are two adjacent triangles in $G$,
		$(u,w_1,w_2)$ and $(v,w_1,w_2)$, where $u$ and $v$ are not adjacent.
		Let $u_1$ and $v_1$ be the neighbors of $u$ and $v$, respectively, distinct from $w_1$ and $w_2$.
		Note that $u_1$ and $v_1$ exist and are distinct, since $G$ has at least $9$ edges and is $2$-connected.
		
		Now, consider the graph $G' = G \setminus \set{w_1,w_2}$. 
		Note first that by Claim~\ref{cl:bridge}, $G'$ is connected.
		It is not isomorphic to $K_{3,3}$, and thus admits a semistrong edge-coloring $\sigma$ with at most $8$ colors due to the minimality of $G$.
		The coloring $\sigma$ induces a partial semistrong edge-coloring of $G$, 
		in which only the edges incident with the two adjacent triangles are non-colored.
		Observe that $w_1w_2$ has at least $6$ available colors, while the other four non-colored edges have at least $4$ available colors each.
		This means that by \cref{thm:hall}, we can color them and hence extend $\sigma$ to all the edges of $G$, a contradiction.
		%
%
	\end{proofclaim}	

	\begin{claim}
		\label{cl:34}
		There is no $4$-cycle adjacent to a triangle in $G$.
	\end{claim}
	
	\begin{proofclaim}
		Suppose to the contrary that there is a $3$-cycle $T=(u,w_1,w_2)$ and a $4$-cycle $F=(v,v',w_1,w_2)$ in $G$.
		By Claims~\ref{cl:multi} and~\ref{cl:33}, the vertex $u$ is distinct from the vertices $v$ and $v'$,
		and $uv$, $uv'$ are not the edges of $G$.
		Consequently, $vv'$ and $uw_1$ may receive the same color in a semistrong edge-coloring.
		
		Consider the graph $G' = G \setminus \set{w_1,w_2}$.
		By Claim~\ref{cl:bridge}, $G'$ is connected, not isomorphic to $K_{3,3}$, and therefore, by the minimality of $G$,
		it admits a semistrong edge-coloring $\sigma$ using at most $8$ colors.
		Consider the partial edge-coloring of $G$ induced by $\sigma$.
		First, we uncolor the edge $vv'$.
		We infer that $vv'$ has at least $2$ available colors,
		$v'w_1$ and $vw_2$ have at least $3$ each,
		$uw_1$ and $uw_2$ have at least $4$ each,
		and $w_1w_2$ has at least $5$ available colors.
		If $A(vv') \cap A(uw_1) \neq \emptyset$, then we can color the edges $vv'$ and $uw_1$ with the same color,
		and the remaining four non-colored edges are colorable by \cref{thm:hall}, a contradiction.
		Thus, we may assume that $A(vv') \cap A(uw_1) = \emptyset$, 
		meaning that $|A(vv') \cup A(uw_1)| \ge 6$. 
		This means that the union of available colors of any five (or six) non-colored edges is of size at least $5$ (or $6$),
		and we can again apply \cref{thm:hall} to extend the coloring to all edges of $G$, a contradiction.		
	\end{proofclaim}

	\begin{claim}
		\label{cl:3no2v}
		No triangle in $G$ is incident with a $2$-vertex.
	\end{claim}
	
	\begin{proofclaim}
		By Claim~\ref{cl:bridge}, every triangle in $G$ is incident with at most one $2$-vertex.
		Now, suppose the contrary and let $T=(v_1,v_2,v_3)$ be a triangle incident with a $2$-vertex $v_1$.
		Then, by the minimality, $G' = G \setminus \set{v_1}$ admits a semistrong edge-coloring $\sigma$ using at most $8$ colors.
		The coloring $\sigma$ induces a partial semistrong edge-coloring of $G$ with only the edges $v_1v_2$ and $v_1v_3$ being non-colored.
		Since both non-colored edges have at least $3$ available colors, we can extend $\sigma$ to all edges of $G$,
		hence obtaining a contradiction.
	\end{proofclaim}

	\begin{claim}
		\label{cl:3}
		There is no triangle in $G$.
	\end{claim}
	
	\begin{proofclaim}
		Suppose to the contrary that $T=(v_1,v_2,v_3)$ is a triangle in $G$. 
		Let $u_1$, $u_2$, $u_3$ be the third neighbors of $v_1$, $v_2$, and $v_3$, respectively. 
		By Claims~\ref{cl:3no2v},~\ref{cl:33}, and~\ref{cl:34}, 
		we have that all three vertices $u_1$, $u_2$, and $u_3$ exist, are distinct, and pairwise non-adjacent.
		We call an edge $u_iv_i$, for every $i \in [3]$, an \textit{incoming edge},
		and every edge in $E(T)$ a \textit{triangle edge}.
		By the minimality, $G \setminus \set{v_1,v_2,v_3}$ admits a semistrong edge-coloring $\sigma$ with at most $8$ colors. 
		We show that $\sigma$ can be extended to all edges of $G$.
		
		Note that there are at least $2$ available colors for every incoming edge and at least $4$ available colors for every triangle edge. 		
		Moreover, note also that an incoming edge must have a color distinct from the color of the opposite triangle edge.
		On the other hand, incoming edges may receive the same colors, since they do not belong to a common $4$-cycle as $u_i$'s are not adjacent.
		
		Suppose first that $|\bigcap_{i=1}^3 \F{u_iv_i}| \ge 1$ and let $1$ be the color available for all the three edges $u_iv_i$. 
		In this case, color all three edges by $1$. There remain at least $3$ available colors for each triangle edge,
		so we can complete the coloring.
		
		Hence, we may assume that $|\bigcap_{i=1}^3 \F{u_iv_i}| = 0$. We consider two subcases.
		\begin{itemize}
			\item[$(i)$] \textit{Two incoming edges have a common available color, say $1 \in \F{u_1v_1} \cap \F{u_2v_2}$.} 
				In this case, we may assume that $\F{u_3v_3}=\set{2,3}$. Color the edges $u_1v_1$ and $u_2v_2$ with $1$.
				Next, consider the available colors of the triangle edges (after coloring the two incoming edges). 
				If $|\bigcup_{e \in E(T)} \F{e} \setminus \set{2,3}| \ge 2$, then there exists a coloring of the four non-colored edges by \cref{thm:hall}.
				
				Thus, we may assume that $\F{v_1v_2} = \F{v_2v_3} = \F{v_1v_3} = \set{2,3,4}$. This in particular means that the
				edges incident with the vertices $u_1$, $u_2$, and $u_3$, distinct from the incoming edges,
				are colored with the colors $\set{5,6,7,8}$. Moreover, for every pair of vertices
				$u_i$ and $u_j$, $i \ne j, i,j\in [3]$, the union of colors on their incident edges (without the color $1$)
				has cardinality $4$. As this is not possible, we reached a contradiction.
			
			\item[$(ii)$] \textit{All the available colors of the incoming edges are distinct} (see Figure~\ref{fig:tri_a} for an illustration).
				\begin{figure}[htp!]
					$$
						\includegraphics{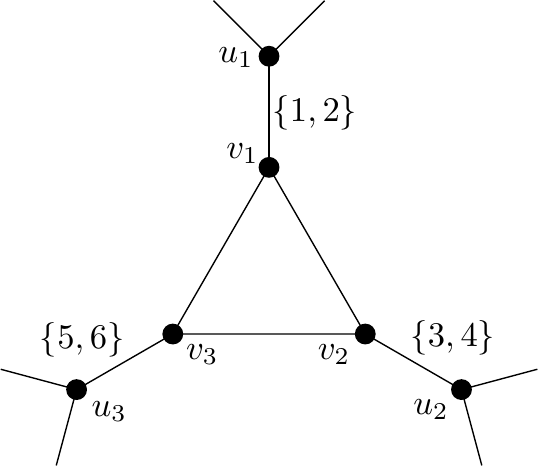}
					$$
					\caption{A triangle with all the available colors of the incoming edges distinct.}
					\label{fig:tri_a}
				\end{figure}
				Suppose first that there exists an available color, say $1$, of some incoming edge such that
				$|\F{e} \setminus \set{1}| \ge 4$ for every triangle edge $e$. 			
				If 
				$$
					\Big| \big( \bigcup_{e \in E(T)} \F{e} \cup \set{3,4,5,6} \big) \setminus \set{1} \Big| \ge 5,
				$$
				then we color $u_1v_1$ by $1$ and extend the coloring to the remaining five edges by \cref{thm:hall}.				
				Thus, we may assume that $\F{v_1v_2} = \F{v_2v_3} = \F{v_1v_3} = \set{3,4,5,6}$. Note that in this case $\C{u_1} = \set{1,7,8}$,
				$\C{u_2} = \set{2,x}$, and $\C{u_3} = \set{2,y}$, where $\set{x,y} = \set{7,8}$. This in particular means that before coloring $u_1v_1$,
				the color $1$ was available for all three triangle edges. Hence, we recolor $u_1v_1$ by $2$, and color the remaining five edges by \cref{thm:hall}.
				
				Therefore, by symmetry, we may assume that every color from $[6]$ is available for some triangle edge. 
				Hence, we infer that
				$|\F{v_1v_2} \cup \F{v_2v_3} \cup \F{v_1v_3}| \ge 6$ and we can color all the six edges by \cref{thm:hall}, a contradiction.
		\end{itemize}
		This establishes the claim.
	\end{proofclaim}

	\begin{claim}
		\label{cl:4no2v}
		No $4$-cycle in $G$ is incident with a $2$-vertex.
	\end{claim}
	
	\begin{proofclaim}
		By Claim~\ref{cl:bridge}, we have that there are at most two $2$-vertices incident with any $4$-cycle.
		We again proceed by contradiction.
		Let $F = (v_1,v_2,v_3,v_4)$ be a $4$-cycle in $G$.
		Suppose first that the vertices $v_1$ and $v_2$ are $2$-vertices (and thus $v_3$ and $v_4$ are $3$-vertices).
		By the minimality of $G$, the graph $G' = G \setminus \set{v_1,v_2}$ admits a semistrong edge-coloring $\sigma$ using at most $8$ colors.
		The coloring $\sigma$ induces a partial coloring of $G$ with only the three edges being non-colored.
		Since each of the three edges has at least $3$ available colors, we can extend $\sigma$ to $G$, a contradiction.
		
		Thus, we may assume that in $F$, there is a pair of opposite $3$-vertices, say $v_1$ and $v_3$,
		and at least one $2$-vertex, say $v_2$. 
		However, if $v_4$ is a $2$-vertex also, then we proceed as in the proof of Claim~\ref{cl:33}.
		Thus, we may assume that $d(v_4) = 3$.
		Let $u_1$ and $u_3$ be the neighbors of $v_1$ and $v_3$, respectively, not incident with $F$.
		In the case when $u_1 = u_3$, we consider a partial coloring of $G$ obtained from a coloring of $G' = G \setminus \set{v_2}$.
		There are at least $2$ available colors for the non-colored edges $v_1v_2$ and $v_2v_3$, and so we can extend the coloring to all edges of $G$, a contradiction.
		
		Therefore, we may assume that $u_1 \neq u_3$. 
		Suppose first that $u_1$ and $u_3$ are adjacent.
		Let $G' = G \setminus \set{v_1,v_2,v_3}$. 
		Clearly, $G'$ is not isomorphic to $K_{3,3}$ and so it admits a semistrong edge-coloring $\sigma$ using at most $8$ colors.
		In the coloring of $G$ induced by $\sigma$, we uncolor the edge $u_1u_3$, obtaining seven non-colored edges.
		The edge $u_1u_3$ has at least $2$ available colors,
		$u_1v_1$ and $u_3v_3$ have at least $3$,
		$v_1v_2$ and $v_2v_3$ have at least $6$,
		and $v_1v_4$ and $v_3v_4$ have at least $4$.
		Now, we color $u_1u_3$ with an available color $\alpha$. 
		Since any edge of $F$ may receive the same color as $u_1u_3$, we remove $\alpha$ from $A(e)$ for no edge of $F$, 
		while the number of available colors for $u_1v_1$ and $u_3v_3$ may decrease by $1$.
		But now, it is easy to see that \cref{thm:hall} applies to the remaining six non-colored edges, and thus $\sigma$ can be extended to all the edges of $G$.
	
		Finally, we may assume that $u_1$ and $u_3$ are not adjacent.
		In this case, by the minimality of $G$, the graph $G' = (G \setminus V(F)) \cup \set{u_1u_3}$ admits a semistrong edge-coloring $\sigma$ using at most $8$ colors.
		In the coloring of $G$ induced by $\sigma$, we color the edges $u_1v_1$ and $u_3v_3$ with the color $\sigma(u_1u_3)$.
		Then, we color the edge $e$ incident with $v_4$ and not being on $F$ with a color which does not appear in $N^2(e)$ 
		(there is at least one such color).
		Now, there remain four non-colored edges with $v_1v_2$ and $v_2v_3$ having at least $4$ available colors
		and $v_1v_4$ and $v_3v_4$ having at least $2$ available colors. 
		We can color them by \cref{thm:hall}, a contradiction.
	\end{proofclaim}

	\begin{claim}
		\label{cl:K23}
		No subgraph of $G$ is isomorphic to $K_{2,3}$.
	\end{claim}
	
	\begin{proofclaim}
		Suppose to the contrary that $G$ contains a subgraph isomorphic to $K_{2,3}$ 
		with bipartition sets $X=\set{x_1,x_2}$ and $Y=\set{y_1,y_2,y_3}$. 
		By Claim~\ref{cl:3}, no two vertices of $Y$ are adjacent. 
		From Claim~\ref{cl:4no2v} it follows that every vertex $y \in Y$ has a neighbor $u_i$ distinct from $x_1$ and $x_2$. 		
		We consider three cases.

		\begin{itemize}
			\item[$(i)$] Some of $u_i$'s are the same. 
				Since $G$ is not isomorphic to $K_{3,3}$, 
				we may assume that $u_1=u_2 \ne u_3$.  
				Let $G'=G \setminus \set{x_1,x_2,y_1,y_2}$. 
				Clearly, $G'$ is distinct from $K_{3,3}$ and, by Claim~\ref{cl:bridge}, it is connected. 
				Thus, by the minimality, $G'$ admits a semistrong edge-coloring $\sigma$ using at most $8$ colors. 
				In $G$, the coloring $\sigma$ induces a partial semistrong edge-coloring with eight non-colored edges; 
				each of the edges $x_1y_1$, $x_1y_2$, $x_2y_1$, and $x_1y_2$ have at least $6$ available colors, 
				and $u_1y_1$, $u_1y_2$, $x_1y_3$, and $x_2y_3$ have at least $5$. 
				Thus, $|A(u_1y_1) \cap A(x_1y_3)| \ge 2$ and $|A(u_1y_2) \cap A(x_2y_3)| \ge 2$, 
				which means that we can color $u_1y_1$ and $x_1y_3$ with a common available color $\alpha$, 
				and $u_1y_2$ and $x_2y_3$ with a common available color distinct from $\alpha$. 
				This is possible, since $u_1$ and $y_3$ are not adjacent. 
				After that, the remaining four non-colored edges still have at least $4$ available colors each, 
				and so we can color them by Theorem~\ref{thm:hall}, hence extending $\sigma$ to all edges of $G$, a contradiction.
		
			\item[$(ii)$] All $u_i$'s are distinct, but some of them are adjacent. 
				By the symmetry, we may assume that $u_1u_2 \in E(G)$. 
				Again, let $G'=G \setminus \set{x_1,x_2,y_1,y_2}$, which is connected by Claim~\ref{cl:bridge} 
				and by the minimality admits a semistrong edge-coloring $\sigma$ using at most 8 colors. 				
				In the corresponding partial coloring of $G$ with eight non-colored edges, 
				we additionaly recolor $u_1u_2$ with a color that does not appear in its $2$-edge-neighborhood  
				(there are at least two such colors) and is distinct from the color of $u_3y_3$.
				Now, we have that $x_1y_1$, $x_1y_2$, $x_1y_3$, $x_2y_1$, $x_2y_2$, and $x_2y_3$ each have at least $5$ available colors, 
				and $u_1y_1$ and $u_2y_2$ have at least $3$. 
				Moreover, $|A(u_1y_1) \cap A(x_1y_2)| \ge 2$, since the two colors from $C(u_2)$ are forbidden for both edges. 
				Analogously, $|A(u_2y_2) \cap A(x_2y_1)| \ge 2$. 
				Next, we color $u_1y_1$ and $x_1y_2$ with a common available color $\alpha$, 
				and $u_2y_2$ and $x_2y_1$ with a common available color distinct from $\alpha$. 				
				After that, the remaining four non-colored edges have at least $3$ available colors each.
				If the union of their available colors is of size at least $4$, then we apply Theorem~\ref{thm:hall},
				and so extend $\sigma$ to all the edges of $G$, a contradiction.
				Thus, we may assume that all four edges have the same set of $3$ available colors.
				But then, we can color $x_2y_2$ with the color of $u_1u_2$, while the three remaining non-colored edges 
				still have $3$ available colors each.
				Now, we can again apply Theorem~\ref{thm:hall} to obtain a semistrong edge-coloring of all edges of $G$, a contradiction.
			
			\item[$(iii)$] All $u_i$'s are distinct and pairwise non-adjacent. 
				In this case, we use $G' = G \setminus (X \cup Y)$, 
				which admits a semistrong edge-coloring $\sigma$ using at most $8$ colors. 
				The three edges $u_iy_i$ (call them {\em incoming edges}) each have at least $2$ available colors,
				and all the remaining non-colored edges have at least $6$ available colors. 				
				In order to simplify our argument, we reduce the number of available colors for the incoming edges;
				in particular, if there are more than $2$ available colors for an incoming edge $e$, 
				then we delete some of them from $A(e)$ to obtain $|A(e)| = 2$.
				
				Now, we proceed as follows.
				We first color $x_1y_1$ with a color $\alpha_1 \notin A(u_1y_1)$. 
				If $\alpha_1 \in A(u_2y_2)$, then we also color $u_2y_2$ with it, 
				and, similarly, if $\alpha_1 \in A(u_3y_3)$, then we also color $u_3y_3$ with it.
				Next, we color $x_2y_1$ with a color $\alpha_2 \notin A(u_1y_1)$, 
				and as above, we color any non-colored incoming edge having $\alpha_2$ as an available color 
				(note that such as edge is clearly distinct from $u_1y_1$). 
				In this way, every non-colored incoming edge retains $2$ available colors,
				and the remaining non-colored edges have at least $4$ available colors each.
				We continue by coloring $x_1y_2$ with a color $\alpha_3 \notin A(u_2y_2)$,
				and coloring any non-colored incoming edge having $\alpha_3$ as an available color.
				Finally, we color $x_2y_3$ with a color $\alpha_4 \notin A(u_3y_3)$,
				and color any non-colored incoming edge having $\alpha_4$ as an available color.
				At this point, all the non-colored edges have at least $2$ available colors.
				We finish by coloring the edge $x_1y_3$ and $x_2y_2$ by distinct available colors, 
				where we use their colors to color any non-colored incoming edge not adjacent to them,
				and complete the coloring by coloring the remaining non-colored incoming edges by their available colors.
				Thus, we colored all the edges of $G$, a contradiction.
		\end{itemize}
	\end{proofclaim}

	\begin{claim}
		\label{cl:4}
		There is no $4$-cycle in $G$.
	\end{claim}
	
	\begin{proofclaim}
		Suppose to the contrary that $F = (v_1,v_2,v_3,v_4)$ is a $4$-cycle in $G$. 
		We denote the third neighbor of $v_i$ by $u_i$, for every $i \in [4]$
		(see Figure~\ref{fig:sq_a}), and call the edges $u_iv_i$ \textit{incoming}. 
		The edges of $F$ are \textit{cycle} edges.
		From Claim~\ref{cl:4no2v} we infer that all $u_i$'s exist,
		and by Claims~\ref{cl:3} and~\ref{cl:K23}, they are all distinct.
		Let $G'$ be the graph obtained from $G$ by removing the vertices of $F$, and adding the edges $u_1u_3$ and $u_2u_4$ 
		(it is possible that parallel edges are introduced).
		Note that $G'$ is not necessarily connected, but in such a case, by Claim~\ref{cl:K23}, none of its components is isomorphic to $K_{3,3}$.		
		Thus, by the minimality, either $G'$ is isomorphic to $K_{3,3}$, meaning that $G$ is a $5$-prism, which admits a coloring with $8$ colors, 
		or there exists a semistrong edge-coloring $\sigma$ of $G'$ using at most $8$ colors. 
		Without loss of generality, we may assume that $\sigma(u_1u_3) = 1$ and $\sigma(u_2u_4) = \alpha$, where $\alpha \in \set{1,2}$.
		The coloring $\sigma$ induces a partial coloring of $G$ with the incoming edges and the cycle edges of $F$ being non-colored.
		We show that $\sigma$ can be extended to all the edges of $G$.
		\begin{figure}[htp!]
			$$
				\includegraphics{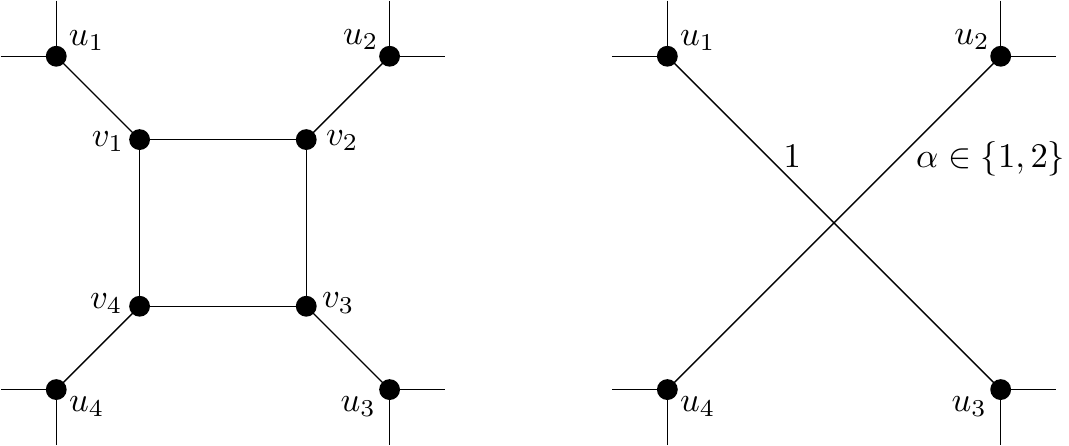}
			$$
			\caption{The $4$-cycle $F$ with the four distinct neighboring vertices (on the left), and the corresponding configuration in the graph $G'$ (on the right).}
			\label{fig:sq_a}
		\end{figure}				
		
		First, observe the following.
		By the construction, if there is a color $\beta \notin C(u_i)$, then $\beta$ can be used for at least one cycle edge incident with $v_i$. 
		In particular, if $\beta \notin C(u_{i+1})$, then $\beta$ can be used for $v_iv_{i+1}$. 
		On the other hand, if $\beta \in C(u_{i-1}) \cap C(u_{i+1})$, then there are two possibilities. 
		\begin{itemize}		
			\item[$(i)$] $u_{i-1}u_{i+1}\notin E(G)$ or $e' = u_{i-1}u_{i+1} \in E(G)$ and $\sigma(e') \ne \beta$. 
				Since $u_{i-1}u_{i+1} \in E(G')$, it follows that any of $v_iv_{i-1}$ and $v_iv_{i+1}$ can be colored with $\beta$.
			\item[$(ii)$] $e' = u_{i-1}u_{i+1} \in E(G)$ and $\sigma(e')=\beta$. 
				In this case, $\beta \notin \sigma(N_{u_{i-1}}^2(u_{i-1}u_{i+1}))$ or $\beta \notin \sigma(N_{u_{i+1}}^2(u_{i-1}u_{i+1}))$, 
				say $\beta \notin \sigma(N_{u_{i-1}}^2(u_{i-1}u_{i+1}))$, 
				and thus $\beta$ can be used for $v_iv_{i-1}$.
		\end{itemize}
		
		Now, we consider two cases. Suppose first that $\alpha = 2$.
		We color $u_1v_1$ and $u_3v_3$ with $1$ (this can be done, since $\sigma(u_1u_3) = 1$ and $u_3v_3 \notin N_{v_1}^2(u_1v_1)$, $u_1v_1 \notin N_{v_3}^2(u_3v_3)$), 
		and $u_2v_2$ and $u_4v_4$ with $2$ (with an analogous reasoning).
		Now, let $A^*(e)$, for $e \in E(F)$, denote the set of colors that can be used to color the edge $e$ without violating the assumptions of the semistrong edge-coloring. 
		Clearly, $A(e) \subseteq A^*(e)$ and so $|A^*(e)| \ge |A(e)| \ge 2$. 
		Moreover, by the observation above, we also infer that $|A^*(v_iv_{i-1}) \cup A^*(v_iv_{i+1})| \ge 4$. 
		Thus, the sets $A^*(e)$ fulfill conditions of Theorem~\ref{thm:hall}, and we can extend $\sigma$ to all edges of $G$, a contradiction.
		
		Next, suppose that $\alpha = 1$.
		We again color $u_1v_1$ and $u_3v_3$ with $1$,
		but in this case, we cannot necessarily color $u_2v_2$ and $u_4v_4$ with the $1$ as the semistrong condition might be violated
		in the case if $1$ appears also in the $2$-edge-neighborhood of $u_1u_3$ in $G'$.
		However, there is at least $1$ available color distinct from $1$ for each of $u_2v_2$ and $u_4v_4$. 
		If there is the same available color for both edges, we color them with it and proceed as in the previous case.
		Hence, we may assume that they are different, say $2$ and $3$, respectively, and so we color $u_2v_2$ with $2$
		and $u_4v_4$ with $3$.		
		We have that, $|A^*(v_1v_2) \cup A^*(v_2v_3)| \ge 4$, 
		$|A^*(v_3v_4) \cup A^*(v_1v_4)| \ge 4$, 
		and $|A^*(e)| \ge 2$ for every $e \in E(F)$. 
		Thus, we can again apply Theorem~\ref{thm:hall} and extend $\sigma$ to all edges of $G$, a contradiction.
		This establishes the claim.
	\end{proofclaim}

	From the above claims it follows that the graph $G$ is a bridgeless subcubic graph with girth at least $5$. 
	
	Now, let $\sigma$ be a proper edge-coloring of $G$ with the minimum number of edges $uv$
	having an edge of color $\sigma(uv) \in N^2_u(uv) \cap N^2_v(uv)$;
	in other words, $uv$ is the middle edge of some path $P_6$ induced by the endvertices of edges colored with $\sigma(uv)$.
	We denote the number of such edges in $G$ by $\iota_6(\sigma)$ and we will refer to them as the {\em bad middle edges}. 
	Additionally, among all such colorings $\sigma$, 
	we choose a coloring with the minimum number of edges of the same color at distance $2$.
	We denote the number of such pairs in $G$ by $\iota_4(\sigma)$ and we refer to such pairs as the {\em bad pairs}.
	Clearly, $\iota_6(\sigma) > 0$ and every bad middle edge is involved in at least two bad pairs. 
	Note also that an edge can be the bad middle edge of several $P_6$s, but we count it only once.
	
	Let $uv$ be a bad middle edge in $G$. 
	We may assume that $\sigma(uv) = 1$.	
	Consider the $2$-edge-neighborhood of the edge $uv$ and label the neighboring vertices as in Figure~\ref{fig:uv}.
	By Claims~\ref{cl:3} and~\ref{cl:4}, all the neighbors of $u$ and $v$ are distinct and non-adjacent.
	There are at most $8$ edges at distance $2$ from $uv$, where at least two of them are colored with $1$.
	By the minimality of $\iota_6(\sigma)$, we cannot recolor $uv$ without introducing at least one new bad middle edge.
	\begin{figure}[htp!]
		$$
			\includegraphics{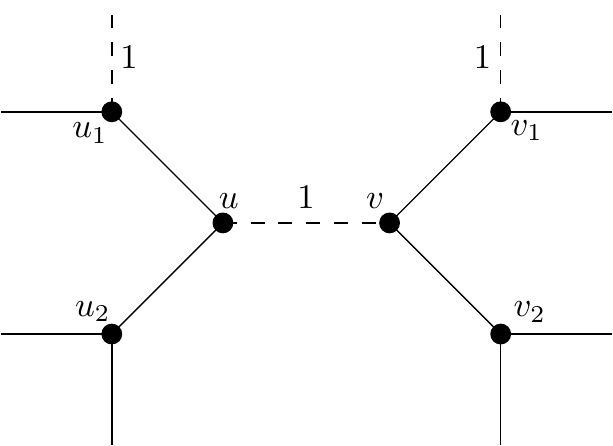}
		$$
		\caption{The edge $uv$ is the bad middle edge
			(the three edges inducing it are indicated with the dashed lines).}
		\label{fig:uv}
	\end{figure}	
	
	Clearly, each of the seven remaining colors must appear in $N^2(uv)$,
	otherwise we recolor $uv$ with the one not appearing in $N^2(uv)$ and thus decrease $\iota_6(\sigma)$.
	Moreover, if some color $\alpha$ appears at most once at distance $2$ from $uv$,
	then we recolor $uv$ with $\alpha$, decrease the number of bad middle edges of color $1$ by at least $1$
	and increase the number of bad middle edges of color $\alpha$ by at most $1$.
	However, while we decrease the number of bad pairs of color $1$ by at least $2$,
	we only increase the number of bad pairs of color $\alpha$ by $1$,
	hence violating the minimality of $\iota_6(\sigma)$ and $\iota_4(\sigma)$.
	Therefore, by a simple counting argument, we infer that the four edges at distance $1$ from $uv$
	must all obtain distinct colors, and the colors on the edges at distance $2$ from $uv$ 
	are different from colors at distance $1$ and
	appear in pairs; consequently, $|N^2(uv) \setminus N(uv)| = 8$.
	
	Suppose now that there is a color, say $6$, on the edges at distance $2$ from $uv$
	appearing at the vertices $u_1$ and $u_2$. 
	By the above argument, we may assume that the edges are colored as in Figure~\ref{fig:uv1}
	and $uv$ cannot be recolored with another color without increasing $\iota_6(\sigma)$ 
	or retaining the value of $\iota_6(\sigma)$ and increasing $\iota_4(\sigma)$.
	\begin{figure}[htp!]
		$$
			\includegraphics{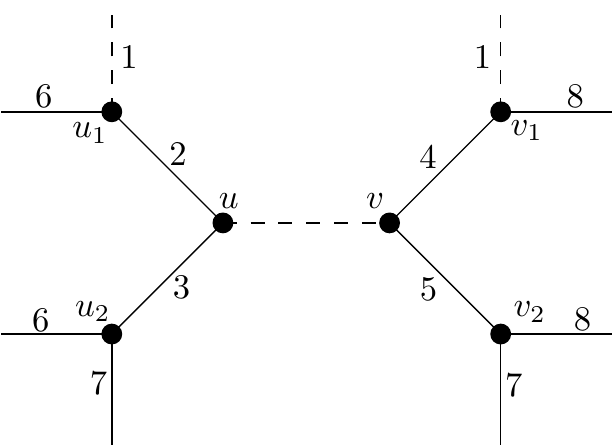}
		$$
		\caption{The edge $uv$ is the bad middle edge and there are two edges at distance $2$ of the same color on one side.}
		\label{fig:uv1}
	\end{figure}		
	This means that there are two edges of color $6$ in $N^2(vv_1)$, 
	otherwise we can recolor $vv_1$ with $6$ and $uv$ with $4$ 
	obtaining a coloring $\sigma'$ with either $\iota_6(\sigma') < \iota_6(\sigma)$ or $\iota_6(\sigma') = \iota_6(\sigma)$ and $\iota_4(\sigma') < \iota_4(\sigma)$.
	Analogously, there are edges of colors $2$ and $3$ in $N_{v_1}^2(vv_1)$.
	But then, we can recolor $vv_1$ with $7$ and $uv$ with $4$, 
	again decreasing $\iota_6(\sigma)$ or retaining the value of $\iota_6(\sigma)$ and decreasing $\iota_4(\sigma)$, a contradiction.
	
	Therefore, we may assume that there are three colors, 
	which, by recoloring $uv$, induce a $P_6$ with $uv$ being the bad middle edge, as depicted in Figure~\ref{fig:sub1}.
	Note that there are precisely two non-isomorphic colorings of the $2$-edge-neighborhood of $uv$, 
	but our argument is analogous for both of them.
	\begin{figure}[htp!]
		$$
			\includegraphics{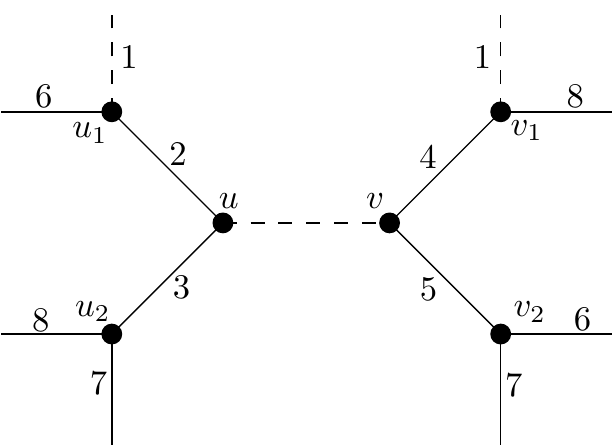}
		$$
		\caption{The edge $uv$ is the bad middle edge and no two edges at distance $2$ of the same color are on one side.}
		\label{fig:sub1}
	\end{figure}
	
	We continue by considering the colors of the edges at distance $3$ from $uv$.
	\begin{claim}
		\label{cl:sub_1n}
		For every edge $xy$ adjacent to $uv$, where $x \in \set{u,v}$, we have 
		$$
			|\sigma(\set{xy} \cup N_y^2(xy))| = 7
		$$
		and
		$$
			\sigma(xz) \notin \sigma(\set{xy} \cup N_y^2(xy)),
		$$
		where $z \notin \set{u,v,y}$.
	\end{claim}
	
	\begin{proofclaim}	
		We start by considering the edge $uu_2$ (we label the vertices as depicted in Figure~\ref{fig:sub2}).
		There already are colors $7$ and $8$ in $N_{u_2}(uu_2)$.
		\begin{figure}[htp!]
			$$
				\includegraphics{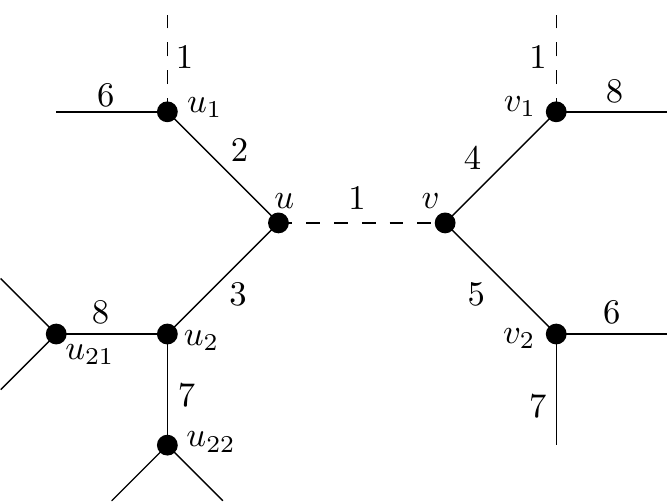}
			$$
			\caption{The neighborhood $N_{u_2}^2(uu_2)$ of $uu_2$ consists of six distinct colors.}
			\label{fig:sub2}
		\end{figure}				
		If no edge from $N_{u_2}^2(uu_2)$ is colored with $1$, 
		then we can set $\sigma(uu_2)=1$ and $\sigma(uv)=3$, 
		hence decreasing $\iota_6(\sigma)$ by at least $1$. 
		So, we may assume that there is an edge colored with $1$ in $N_{u_2}^2(uu_2)$.		
		Similarly, there is an edge colored with $6$ in $N_{u_2}^2(uu_2)$, for otherwise we set $\sigma(uu_2)=6$ and $\sigma(uv)=3$, 
		decreasing $\iota_4(\sigma)$ by $1$, and retaining or decreasing $\iota_6(\sigma)$ as we lose one bad middle edge of color $1$
		and introduce at most one bad middle edge of color $6$. 
		An analogous argument implies that also $4,5 \in N_{u_2}^2(uu_2)$ and hence 
		$$
			\sigma(N_{u_2}^2(uu_2) \setminus N_{u_2}(uu_2)) = \set{1,4,5,6}
		$$		
		Note that the above argumentation works regardless if all the edges in $N_{u_2}^2(uu_2)$ are distinct 
		from the edges in $N^2(uv)$ or not (none of them is in $N(uv)$ due to the girth condition), 
		since we only recolor the edge $uv$ with $3$, which does not appear in $N^2(uv)$, 
		and recolor the edge $uu_2$ with a color which does not appear in $N_{u_2}^2(uu_2)$.
		
		Furthermore, by the symmetry, we also have that 
		\begin{align*}
			\sigma(N_{u_1}^2(uu_1) \setminus N_{u_1}(uu_1)) &= \set{4,5,7,8}, \\			
			\sigma(N_{v_1}^2(vv_1) \setminus N_{v_1}(vv_1)) &= \set{2,3,6,7}, \\
			\sigma(N_{v_2}^2(vv_2) \setminus N_{v_2}(vv_2)) &= \set{1,2,3,8}.
		\end{align*}
		This establishes the claim.
	\end{proofclaim}	
		
	We continue with an analysis of a possible arrangement of colors also on the edges at distance $3$ from $uv$ (depicted in Figure~\ref{fig:sub5}).
	Note that the colors need not be all distinct.
	\begin{figure}[htp!]
		$$
			\includegraphics{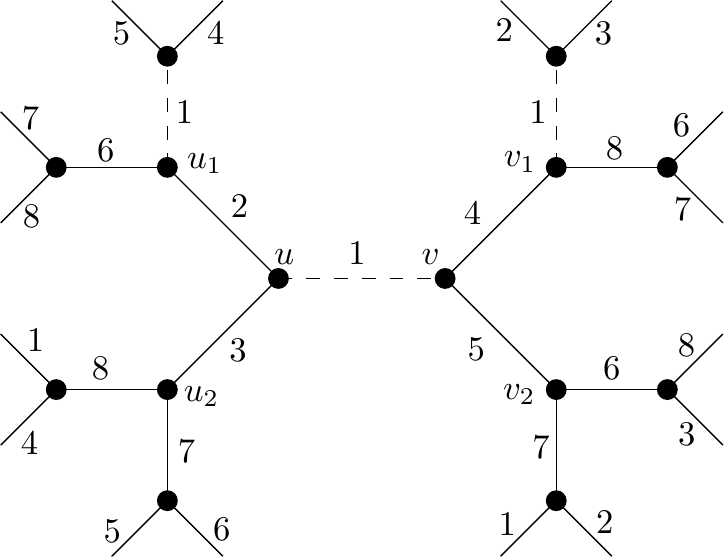}
		$$
		\caption{Possible neighborhoods of the vertices $u_1$, $v_1$, $u_2$, and $v_2$.}
		\label{fig:sub5}
	\end{figure}			
	
	In the next claim, we show that using a similar argument as in the proof of Claim~\ref{cl:sub_1n}, 
	we can determine the colors in $2$-edge-neighborhoods of the edges at distance more than $2$ from $uv$.

	\begin{claim}
		\label{cl:sub_2}
		Let $w_0w_1\dots w_k$ be an induced path on $k+1$ vertices in $G$, for some integer $k \ge 2$, 
		with $w_0 = v$ and $w_1 = u$ (or $w_0 = u$ and $w_1 = v$). 
		Then for every $j$, $2 \le j \le k$, we have that $d(w_j) = 3$ 
		and $\sigma(N^2_{w_j}(w_{j-1}w_j)) = [8] \setminus \set{\sigma(w_{j-1}w_j), \sigma(w_{j-1}x_{j-1})}$,
		where $x_{j-1}$ is the neighbor of $w_{j-1}$ distinct from $w_{j-2}$ and $w_j$.
	\end{claim}	

	\begin{proofclaim}	
		Note that the case for $k=2$ already follows from Claim~\ref{cl:sub_1n}.
		
		We proceed by contradiction. 
		Let $k \ge 3$ be the least integer such that there is an induced path 
		$P = w_0w_1\dots w_k$ in $G$, for which the claim does not hold. 
		Hence, for every $j$, $2 \le j \le k-1$, we have $d(w_{j-1}) = d(w_j) = 3$ and 
		$\sigma(N^2_{w_j}(w_{j-1}w_j)) = [8] \setminus \{\sigma(w_{j-1}w_j), \sigma(w_{j-1}x_{j-1})\}$.
		
		Suppose now that there is a color 
		$\alpha\in [8] \setminus \{\sigma(w_{k-1}w_k), \sigma(w_{k-1}x_{k-1})\}$ 
		which is not in $\sigma(N^2_{w_k}(w_{k-1}w_k))$. 
		Then, for every $j \in \{1,\dots,k-1\}$, we recolor $w_{j-1}w_j$ with $\sigma(w_jw_{j+1})$,
		and finally we recolor $w_{k-1}w_k$ with $\alpha$, obtaining a proper edge-coloring $\sigma'$. 
		As the path $P$ is induced, there is no edge outside of $P$ joining two vertices of $P$. 
		If after recoloring, we introduce a new pair of edges at distance $2$ 
		colored with the same color (note that at most one such pair may appear), 
		then exactly one of these edges belongs to $P$ or it is the pair $w_{k-1}w_k$, $w_{k-3}w_{k-2}$. 
		
		After the recoloring, the edge $w_0w_1$ is not a bad middle edge anymore,
		and at most one new bad middle edge $e'$ was created; 
		$e'$ (if it exists) is either incident with $x_{k-1}$ or $w_{k-2}$. 
		Thus, $\iota_6(\sigma') \le \iota_6(\sigma)$. 
		Moreover, $\iota_4(\sigma') < \iota_4(\sigma)$, 
		since exactly one new pair ($w_kw_{k-1}$ and $e'$) was created, 
		and two pairs (both with the edge $w_0w_1$) were destroyed, a contradiction.
	\end{proofclaim}

%
%

	We finalize the proof of the theorem by recoloring some of the edges of $G$ 
	in order to obtain a contradiction in terms of the assumptions on $\sigma$.
	Consider the labeling of the vertices as depicted in Figure~\ref{fig:subFinA}.
	\begin{figure}[htp!]
		$$
			\includegraphics{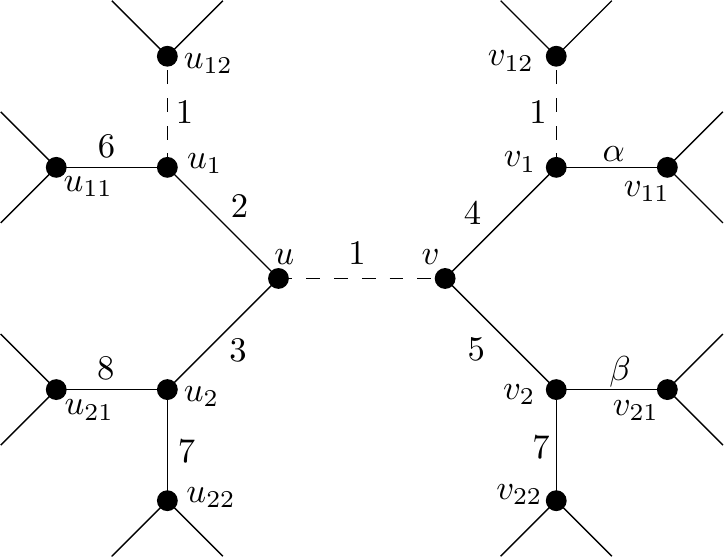}
		$$
		\caption{Colors of the edges around $uv$ before the final step of the proof.}
		\label{fig:subFinA}
	\end{figure}
	Without loss of generality, we may also assume the coloring of the edges as given in the figure,
	where $\set{\alpha,\beta} = \set{6,8}$.

	We first establish some additional properties of $G$.
	\begin{claim}
		\label{cl:distance}
		$d(u,v_{22}) \ge 3$, $d(u,v_{21}) \ge 3$, and $d(v,u_{22}) \ge 3$.
	\end{claim}
	\begin{proofclaim}
		Suppose, to the contrary, that $d(u,v_{22}) < 3$.
		Then, since $G$ has girth at least $5$, we have $d(u,v_{22}) = 2$,
		meaning that $v_{22}$ is adjacent to $u_1$ or $u_2$;
		namely $v_{22} = u_{11}$, $v_{22} = u_{12}$, $v_{22} = u_{21}$, or $v_{22} = u_{22}$.
		We consider each of the four cases separately.
		
		Suppose first that $v_{22} = u_{11}$.
		Then, by Claim~\ref{cl:sub_1n} with $w_k = v_2$, the vertex $v_{22}$ is not incident with edges of color $4$ or $5$.
		Thus, by Claim~\ref{cl:sub_1n} with $w_k = u_1$, $C(v_{22}) = \set{6,7,8}$.
		But this is a contradiction with Claim~\ref{cl:sub_1n} with $w_k = v_2$, since $\beta \in \set{6,8}$.
		
		Next, suppose that $v_{22} = u_{12}$.
		Let $v_{22}'$ be the neighbor of $v_{22}$ distinct from $u_1$ and $v_2$.
		By Claim~\ref{cl:sub_1n} applied twice, with $w_k = v_2$ and $w_k = u_1$, we infer that $\sigma(v_{22}v_{22}') = 8$ 
		and consequently $\beta = 6$.
		If $3 \notin C(v_{22}')$, then by Claim~\ref{cl:sub_1n} with $w_k = u_2$,
		inferring that there is no edge of color $3$ in $N^2(uu_2)$,
		we may recolor $u_1v_{22}$ with $3$ and so decrease $\iota_6(\sigma)$ by $1$. 
		Therefore, $3 \in C(v_{22}')$, meaning that at least one of the colors $4$ and $5$ is not in $C(v_{22}')$.
		Without loss of generality, we may assume $5 \notin C(v_{22}')$ (otherwise we swap the colors of $vv_1$ and $vv_2$).
		Now, we set $\sigma(uv) = \sigma(v_2v_{22}) = 5$, $\sigma(vv_2) = 1$, and $\sigma(u_1v_{22}) = 7$.
		Note that such a recoloring reduces $\iota_6(\sigma)$ by $1$, since colors $5$ and $7$ do not introduce 
		any new $P_6$s, a contradiction.
		
		Suppose now that $v_{22} = u_{21}$.
		Then, by Claim~\ref{cl:sub_1n} with $w_k = u_2$, $\sigma(N_{u_2}^2(uu_2) \setminus N_{u_2}(uu_2)) = \set{1,4,5,6}$,
		but this is not possible, since $\sigma(v_2v_{22}) = 7 \in C(N_{u_2}^2(uu_2) \setminus N_{u_2}(uu_2))$, a contradiction.
		
		Finally, suppose that $v_{22} = u_{22}$. 
		In this case, $\sigma(u_2v_{22}) = \sigma(v_2v_{22})$ violating the assumption that $\sigma$ is a proper edge-coloring, a contradiction.
		
		An analogous reasoning can be used to prove that $d(u,v_{21}) \ge 3$ and $d(v,u_{22}) \ge 3$.
		This establishes the claim.		
	\end{proofclaim}
	
	Now, recolor 
	$uv$ with $7$,
	$u_2u_{22}$ and $vv_2$ with $3$, and	
	$uu_{2}$ with $5$.
	There is one edge, $v_2v_{22}$, colored with $7$ in $N^2(uv)$, 
	and, by Claim~\ref{cl:sub_2}, the only edge colored with $7$ in $N^2(v_2v_{22})$ is $uv$.
	By Claim~\ref{cl:sub_1n}, there is one edge $e_1=xy$ colored with $5$ in $N^2(uu_2)$, 
	where $x \in \set{u_{21},u_{22}}$.
	Note that $y \notin \set{v_{11},v_{12},v_{21},v_{22}}$, by Claim~\ref{cl:sub_1n} applied twice, with $w_k = v_1$ and $w_k = v_2$.
	This means that $d(v,y) \ge 3$,
	and thus, by Claim~\ref{cl:sub_2}, the only edge of color $5$ in $N^2(e_1)$ is $uu_2$.
	Finally, by Claim~\ref{cl:sub_1n} with $w_k = v_2$, there is one edge $e_2$ colored with $3$ in $N_{v_2}^2(vv_{2})$,
	and by Claims~\ref{cl:sub_2} and~\ref{cl:distance} there is one edge $e_3$ colored with $3$ in $N_{u_{22}}^2(u_2u_{22})$.
	By Claim~\ref{cl:sub_2}, none of $e_2$ and $e_3$ is a bad middle edge (and we are done), unless $e_2 = e_3$.
	Thus, for the rest of the proof, assume that $e_2 = e_3$ and $\sigma(e_2) = 3$.
	We may assume that $e_2$ is incident with $v_{22}$ (the case when $e_2$ is incident with $v_{21}$ proceeds with an analogous reasoning), 
	and so let $e_2 = v_{22}z_1$ (see Figure~\ref{fig:subFinB}).
	\begin{figure}[htp!]
		$$
			\includegraphics{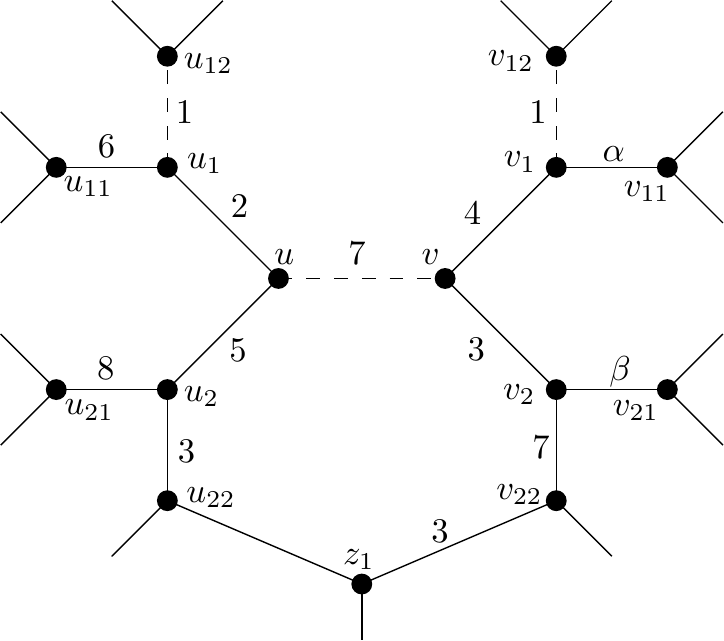}
		$$
		\caption{Colors of the edges around $uv$ after recoloring the four edges and $e_2$ being a bad middle edge.}
		\label{fig:subFinB}
	\end{figure}
	
	In this case, we additionally recolor $u_2u_{21}$ with $3$
	and $u_2u_{22}$ with $8$.
	Note that by Claim~\ref{cl:sub_2}, there is no edge of color $8$ at distance $2$ from $u_2u_{22}$,
	and there is an edge $e_4$ of color $3$ in $N_{u_{21}}^2(u_2u_{21})$.
	By Claim~\ref{cl:sub_1n} and since $e_2$ is incident with $v_{22}$, $3 \notin C(v_{21})$,
	and so $e_4 \ne vv_2$.
	Therefore, we are done, unless $e_4=x'y'$ is a bad middle edge.
	But in this case, the path $vuu_2u_{21}x'y'$ is induced,
	meaning that by Claim~\ref{cl:sub_2} with $w_k = y'$, there is no edge of color $3$ 
	in $N_{y'}^2(e_4)$.
	
	This establishes the proof.
\end{proof}

\section{Conclusion}
\label{sec:con}

We believe that the bounds presented in this paper, although tight, can be improved.
In particular, we are not aware of any family of connected graphs $G$, other than the complete bipartite graphs,
that would attain the bound $\chiss{G} = \Delta(G)^2$.
Therefore, we propose the following.
\begin{conjecture}
	For every connected $G$, distinct from $K_{n,n}$, it holds that
	$$
		\chiss{G} \le \Delta(G)^2 - 1\,.
	$$
\end{conjecture}

Similarly, the $5$-prism is, up to our knowledge, the only connected subcubic graph with the semistrong chromatic index $8$.
Based on that and computational verification of subcubic graphs on small number of vertices,
we also propose the next conjecture.
\begin{conjecture}
	For every connected graph $G$ with maximum degree $3$, distinct from $K_{3,3}$ and the $5$-prism, we have
	$$
		\chiss{G} \le 7\,.
	$$
\end{conjecture}

On the other hand, there are infinitely many bridgeless subcubic graphs with the semistrong chromatic index at least $7$.
Namely, consider the graph $H$ obtained by taking two copies of $K_{2,3}$ and adding two edges 
connecting distinct $2$-vertices from each of two copies (see Figure~\ref{fig:k23x2}).
This graph contains only one semistrong matching of size at least $3$ 
consisting of two edges added between the copies of $K_{2,3}$ and one nonadjacent edge. 
For the remaining $11$ edges we need $6$ additional colors. 
Since $H$ contains two $2$-vertices, we can append it to any other bridgeless subcubic graph with at least
two $2$-vertices, hence obtaining an infinite number of graphs with the semistrong chromatic index at least $7$.
\begin{figure}[htp!]
	$$
		\includegraphics{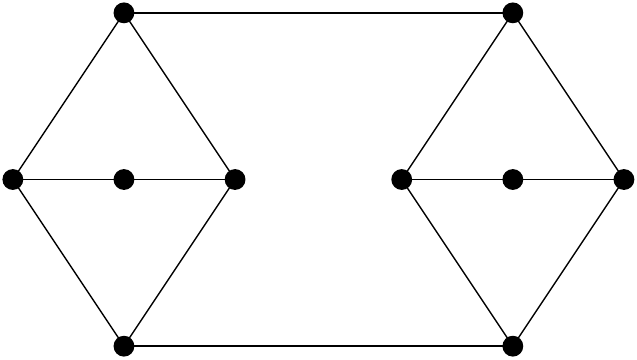}
	$$
	\caption{The graph $H$ with $\chiss{H} = 7$.}
	\label{fig:k23x2}
\end{figure}

There are many other important graph classes for which the semistrong edge-coloring has not been studied specifically yet,
e.g., planar graphs.
In~\cite{FauGyaSchTuz90}, it is proved that for a strong edge coloring of a planar graph $G$ at most $4\Delta(G) + 4$ colors are needed, 
and there are planar graphs $G$ with $\chis{G} = 4\Delta(G) - 4$ for any $\Delta(G) \ge 2$.
Note that these examples have $\chiss{G} \le 2\Delta(G)$. But as we showed with $K_{3,3}$ and the $5$-prism, for cubic graphs, we need $9$ and $8$ colors.
So, one may ask, what is happening when the maximum degree is larger.
We propose the following.
\begin{problem}
	For a given maximum degree, determine the tight upper bound for the semistrong chromatic index of planar graphs.
\end{problem}
\begin{conjecture}
	There is a (small) constant $C$ such that for any planar graph $G$, it holds that
	$$
		\chiss{G} \le 2\Delta(G) + C\,.
	$$
\end{conjecture}

\smallskip
Finally, let us briefly mention another edge-coloring variation, 
which is closely related to the semistrong edge-coloring, yet different.
Baste and Rautenbach~\cite{BasRau18}, motivated by the results of Goddard et al.~\cite{GodHedHedLas05},
introduced the {\em $r$-chromatic index $\chi_r'(G)$} as the minimum number of $r$-degenerate matchings into which
the edge set of a graph $G$ can be decomposed. 
An {\em $r$-degenerate matching} is a matching $M$ such that the induced graph $G[V(M)]$ is $r$-degenerate.
Clearly,
$$
	\chi'(G) \le \chi_r'(G) \le \chis{G}\,.
$$

Since semistrong matchings are not necessarily $1$-degenerate 
and $1$-degenerate matching are not necessarily semistrong,
there is no direct correspondence between the two edge-colorings.
In particular,
for the $5$-path $P_6$, we have $\chi_1'(P_6) = 2 < \chiss{P_6} = 3$,
and for the graph $H$ being a triangle with a pending edge incident with every vertex,
we have $\chiss{H} = 4 < \chi_1'(H) = 5$.

\paragraph{Acknowledgment.} 
B.~Lu\v{z}ar was partially supported by the Slovenian Research Agency Program P1--0383, and Projects J1--1692 and J1--3002.
M.~Mockov\v{c}iakov\'{a} acknowledged the financial support from the project GA20--09525S of the Czech Science Foundation.
R.~Sot\'{a}k was supported by APVV--19--0153 and VEGA 1/0574/21.

\paragraph{Final remark.} 
Before we finished preparation of this draft, Martina passed away, taken by a treacherous disease.
We would like to dedicate this work to her memory.

\bibliographystyle{plain}
\bibliography{mainBib}

\end{document}